\documentclass[a4paper,12pt]{article}
\usepackage{amssymb,amsthm,amsmath,amsfonts}
\usepackage{graphicx}
\usepackage{cite}
\usepackage{color}

\newtheorem{theorem}{Theorem}[section]

\newtheorem{corollary}[theorem]{Corollary}
\newtheorem{proposition}[theorem]{Proposition}

\theoremstyle{definition}
\newtheorem{definition}[theorem]{Definition}

\theoremstyle{remark}
\newtheorem{remark}{Remark}

\begin{document}

\title{A fractional Borel-Pompeiu type formula for holomorphic functions of two complex variables}
\small{
\author
{Jos\'e Oscar Gonz\'alez-Cervantes$^{(1)}$ and Juan Bory-Reyes$^{(2)\footnote{corresponding author}}$}
\vskip 1truecm
\date{\small $^{(1)}$ Departamento de Matem\'aticas, ESFM-Instituto Polit\'ecnico Nacional. 07338, Ciudad M\'exico, M\'exico\\ Email: jogc200678@gmail.com\\$^{(2)}$ {SEPI, ESIME-Zacatenco-Instituto Polit\'ecnico Nacional. 07338, Ciudad M\'exico, M\'exico}\\Email: juanboryreyes@yahoo.com
}
\maketitle
\begin{abstract}
\small{
The present paper is a continuation of our work \cite{BG1}, where we introduced a fractional operator calculus related to a fractional $\psi-$Fueter operator in the one-dimensional Riemann–Liouville derivative sense in each direction of the quaternionic structure, that depends on an additional vector of complex parameters with fractional real parts. 

This allowed us also to study a pair of lower order fractional operators and prove the associated analogues of both Stokes and Borel-Pompieu formulas for holomorphic functions in two complex variables.}
\end{abstract} 
formula; Holomorphic functions of several complex variables.  

\section{Introduction} 
Fractional calculus, involving derivatives-integrals of arbitrary real or complex order, is the natural generalization of the classical calculus, which in the latter years became a well-suited tool by many researchers working in several branches of science and engineering.

The interest in the subject has been growing continuously during the last few decades because of numerous applications in diverse fields of science and engineering, see \cite{GM, KST, MABZ, OS, O, P, MR, T, SKM}. For a brief history and exposition of the fundamental theory of fractional calculus we refer the reader to \cite{Ro}.

The development of a fractional hyperholomorphic function theory represents a very recent topic of research, see \cite{CDOP, DM, FRV, KV, PBBB, V} and the references given there. In particular, the interest for considering fractional Laplace and Dirac Operator is devoted in \cite{Ba, Be, FV1, FV2}.

Nowadays, quaternionic analysis is regarded as a broadly accepted branch of classical analysis offering a successful generalization of complex holomorphic function theory, the most renowned examples are Sudbery's paper\cite{sudbery} and the books \cite{GS1, GS2, KS, K}. It relies heavily on results on functions defined on domains in $\mathbb R^4$ with values in the skew field of real quaternions $\mathbb H$ associated to a generalized Cauchy-Riemann operator (the so-called $\psi-$Fueter operator) by using a general orthonormal basis in $\mathbb R^4$ (to be named structural set) $\psi$ of $\mathbb H^4$, see, e.g., \cite{Na, No1, No2}. 

It is well-known that each of the following four theories: holomorphic functions of two complex variables, vector analysis, time-harmonic electromagnetic, and time-harmonic spinor fields can be embedded into the quaternionic analysis associated to the $\psi-$Fueter operator (see \cite{S1} for more details). Such embeddings are rather fruitful for these conventional theories, which allows both to obtain new results for them and to give a new explanation or interpretation to the previously known facts. This point of view can be expected to be quite promising and far-reaching,  hence the interest of possible extension to the fractional setting. A first step in this direction will be taken here as the main and novel interest of this work.

In \cite{BG1}, we developed a fractional operator calculus related to a fractional $\psi-$Fueter operator in the one-dimensional Riemann–Liouville derivative sense, in each direction of the quaternionic structure, that depends on additional vector of complex parameters with fractional real parts. This fractional construction allows us to introduce a pair of lower order fractional operators and a reproducing function, as a weak analog of the Cauchy kernel. As a main results, we prove the associated analogues of both Stokes and Borel-Pompieu formulas for holomorphic functions in two complex variables. The obtained results exhibit an interesting and original approach for a fractional calculus of holomorphic functions of two complex variables.

\section{Preliminaries} 
This section contains a brief summary of basic definitions and facts on fractional calculus and the theory of $\psi$-hyperholomorphic functions.  

\subsection{Standard definition of and results on Riemann-Liouville fractional integro-differential operators}
There are different definitions of fractional derivatives. One of the most popular (even though it has disadvantages for applications to real world problems) is the Riemann–Liouville derivative (see, e.g., \cite{MR}). For completeness, we recall the key definitions and results on Riemann-Liouville fractional integro-differential operators

Given $\alpha\in \mathbb C$ with $\Re \alpha> 0$, let us recall that the Riemann-Liouville integrals of order $\alpha$ of a  $f  \in L^1([a, b], \mathbb R)$, with  $-\infty <a  < b< \infty$, on the left and on the right, are defined  by 
$$({\bf I}_{a^+}^{\alpha} f)(x) := \frac{1}{\Gamma(\alpha)} \int_a^x \frac{f(\tau)}{(x-\tau)^{1-\alpha}} d\tau, \quad \textrm{with}  \quad x > a$$
and
$$({\bf I}_{b^-}^{\alpha} f)(x) := \frac{1}{\Gamma(\alpha)} \int_x^b \frac{f(\tau)}{(\tau-x)^{1-\alpha}} d\tau, \quad \textrm{with}  \quad x < b,$$
respectively. 

What is more, let $n=[\Re \alpha]+1$, where $[\cdot]$ means the integer part of $\cdot$ and $f\in AC^n([a, b], \mathbb R)$; i.e., the class of functions $f$ which are continuously differentiable on the segment $[a, b]$ up to the order $n-1$ and $f^{(n-1)}$ is supposed to be absolutely continuous on $[a, b]$. The fractional derivatives in the Riemann-Liouville sense, on the left and on the right, are defined by 
\begin{align}\label{FracDer} 
(D _{a^+}^{\alpha} f)(x):= \frac{d}{dx^n} \left[ ({\bf I}_{a^+}^{n-\alpha} f)(x)\right]
\end{align}
and
\begin{align} \label{FracDer1}
(D _{b^-}^{\alpha} f)(x):= (-1)^n\frac{d}{dx^n}\left[({\bf I}_{b^-}^{n-\alpha}f)(x)\right] 
\end{align}
respectively. It is worth noting that the derivatives in (\ref{FracDer}), (\ref{FracDer1}) exist for $f\in AC^n([a, b], \mathbb R)$. Fractional Riemann-Liouville integral and derivative are linear operators.

Fundamental theorem for Riemann-Liouville fractional calculus \cite{CC} shows that 
\begin{align}\label{FundTheorem}
(D_{a^+}^{\alpha} {\bf I}_{a^+}^{\alpha}f)(x)=f(x) \quad  \textrm{and} \quad (D _{b^-}^{\alpha}  {\bf I}_{b^-}^{\alpha} f)(x) = f(x).
\end{align}

Let us mention an important property of the fractional Riemann-Liouville integral and derivative, see \cite[pag. 1835]{VTRMB}.
\begin{proposition}
\begin{equation}\label{cte}
( D _{a^+}^{\alpha} 1)(x)=\frac{(x-a)^{-\alpha}}{\Gamma[1-\alpha]}  , \quad \forall   x\in [a, b].
\end{equation}
\end{proposition}
\subsection{Rudiments of $\psi$-hyperholomorphic functions}
Consider the skew field of real quaternions $\mathbb H$ with its basic elements $1, i, j, k$. Thus any element $q$ from $\mathbb H$ is of the form $q=x_0+x_{1}i+x_{2}j+x_{3}k$, $x_{s}\in \mathbb R, s= 0,1,2,3$, where $i^{2}=j^{2}=k^{2}=-1$ and satisfy the following multiplication rules: 
$$i\,j=-j\,i=k;\   j\,k=-j\,k=i\  \mbox{and}\   k\,i=-k\,i=j.$$

Let us denote by $\textsf{i}$ the imaginary unit of the complex plane $\mathbb C({\textsf{i}})$. By definition, $\textsf{i}$ commutes
with all the quaternionic imaginary units $i, j, k$. We will identify $\textsf{i}$ with $i$ when no confusion can arise. 

For $q\in \mathbb H$ we define the quaternionic conjugation mapping: 
$$q\rightarrow {\overline q}:=x_0-x_{1}i-x_{2}j-x_{3}k.$$ 
In this way it is easy seen that $q\,{\overline q}={\overline q}\,q=|q|^{2}$ and  $\overline {pq}={\overline q}\,\,{\overline p}$ for $p,q\in \mathbb H$. We have for $a\in \mathbb C({i})=\{x+iy  \ \mid \ x,y\in\mathbb R\} $ and $\overline a$ its complex conjugate, that $a\,j=j\,\overline a$.

We embed the usual complex linear space $\mathbb C({i})^2$ into $\mathbb H$ by means of the mapping that associates the pair
$$(z_1, z_2)= (x_0+{i}x_{1}, x_{2}+{i}x_{3})$$
with the quaternion
$$q=z_{1}+z_{2}j,\ \  z_{1},z_{2}\in \mathbb C({i}).$$
The above embedding means that the set of quaternions of the form $q=z_{1}+z_{2}j$ is endowed both with a component-wise addition and with the associative multiplication is then another way of stating $\mathbb H$. In particular, two quaternions $q=z_1+z_2j$ and $\xi=\zeta_1+\zeta_2j$ are multiplied according the rule:
$$q\,\xi=(z_1\zeta_1-z_2\bar\zeta_2)+(z_1\zeta_2+z_2\bar\zeta_1)j.$$ 
The quaternion conjugation gives: $\overline{z_1+z_2j}:= \bar z_1-z_2j$ and $q\,{\overline q}={\overline q}\, q=|z_{1}|^{2}+|z_{2}|^{2}$. 

We will identify $\mathbb H$ with $\mathbb C({i})^2$ (or $\mathbb R^4$) by mean of the mapping
\begin{equation} \label{mapping}
x_0+x_{1}i+x_{2}j+x_{3}k\rightarrow (x_0+ix_1)+(x_2+ix_3)j \rightarrow (x_0,x_1,x_2,x_3).
\end{equation}

The set of complex quaternions is given by  
$$\mathbb H (\mathbb C(\textsf{i})) =\{q=q_1+ \textsf{i}  q_2 \ \mid \ q_1,q_2 \in \mathbb H\}.$$
The main difference to the real quaternions is that not all non-zero complex quaternions are invertible. There are zero-divisors. 

By $\mathbb C ( i)(\mathbb C(\textsf{i}))$ we mean 
$$ \mathbb C ( i) (\mathbb C(\textsf{i})) = \{w_1+  \textsf{i} w_2 \ \mid \ w_1, w_2 \in \mathbb C({i})\}.$$

Let us recall that $\mathbb H$ is embedded in $\mathbb H(\mathbb C (\textsf{i}))$ as follows:
$$\mathbb H =\{q=q_1+ \textsf{i} q_2 \  \in \mathbb H(\mathbb C(\textsf{i}))  \ \mid \  q_1,q_2 \in \mathbb H \ \ \textrm{and} \ \ q_2=0\}.$$

Quaternionic analysis is regarded as a broadly accepted branch of classical analysis offering a successful generalization of complex holomorphic function theory. It relies heavily on results on functions defined on domains in $\mathbb R^4\cong\mathbb C({i})^2$ with values in $\mathbb H$ or $\mathbb H (\mathbb C(\textsf{i}))$ associated to a generalized Cauchy-Riemann operator by using a general orthonormal basis in $\mathbb R^4$, a so-called structural set of $\mathbb H^4$. The concept of structural set goes back as far as \cite{Na, No1, No2}. 

This theory is centered around the concept of $\psi-$hyperholomorphic functions, i.e., the collection of all null solutions of the so-called $\psi-$Fueter operator associated to the structural set $\psi=\{1, \psi_1, \psi_2, \psi_3\}\in \mathbb H^4$:
\[{}^\psi\mathcal D= \frac{\partial}{\partial x_0} + \psi_1 \frac{\partial}{\partial x_1} +  \psi_2 \frac{\partial}{\partial x_2} + \psi_3 \frac{\partial}{\partial q^3},\]
which represents a natural generalization of the complex Cauchy-Riemann operator, see \cite{MS, S1, S2, SV1, SV2} for more details.

The theory of $\psi$-hyperholomorphic functions offers yet a particularly suitable framework for the treatment of some conventional theories \cite{S1}. In particular, it is shown that holomorphic functions of two complex variables can be embedded into one of the above $\psi$-hyperholomorphic ones. 

Let establish a one-to-one correspondence between $\psi-$hyperholomorphic functions $f$ in $\Omega\subset\mathbb R^4\cong\mathbb C({i})^2$ and a pair of complex-valued functions:  
\[
f=f^0+f^{1}i+f^{2}j+f^{3}k=(f^0+if^1)+(f^2+if^3)j=f_1+f_2 j, 
\]
where $f^s, s= 0,1,2,3,$ are $\mathbb R$-valued functions. Properties as continuity, differentiability, integrability and so on, which as ascribed to $f$ have to be posed by all components $f^s$. We will follow standard notation, for example $C^{1}(\Omega, \mathbb H)$ denotes the set of continuously differentiable $\mathbb H$-valued functions defined in $\Omega$. 

Given $0\leq \theta < 2\pi$, set $\psi=\{1, i, i e^{i\theta} j, e^{i\theta} j\}$  and introduce the corresponding $\psi-$Fueter operator
\[{}^\theta\mathcal D:= \frac{\partial}{\partial x_0} + i \frac{\partial}{\partial x_1} + i e^{i\theta}j \frac{\partial}{\partial x_2} +  e^{i\theta} j \frac{\partial}{\partial x_3},\]
to be written in complex form  
\[{}^\theta\mathcal D=2\left\{\frac{\partial}{\partial \bar z_1} + i
e^{i\theta}\frac{\partial}{\partial z_2}j\right\} =2\left\{
\frac{\partial}{\partial \bar z_1} + i
e^{i\theta}j\frac{\partial}{\partial \bar z_2}\right\}
\]
defined on $C^{1}(\Omega, \mathbb H)$. 

For ${}^\theta\mathcal D $ acting on the right we will denote it by ${}^\theta\mathcal D_r$.  These operators operators decompose the four-dimensional Laplace operator:
$${}^\theta\mathcal D\circ \overline{{}^\theta\mathcal D}=\overline{ {}^\theta\mathcal D} \circ {}^\theta\mathcal D =
\overline{{}^\theta\mathcal D_r }\circ {}^\theta\mathcal D_r =
{}^\theta\mathcal D_r \circ \overline{{}^\theta\mathcal
D_r} =\bigtriangleup_{\mathbb R^4}.$$
The elements of the sets ${}^\theta\mathfrak M(\Omega, \mathbb H)=ker \ {}^\theta\mathcal D$ and ${}^\theta \mathfrak M_r(\Omega, \mathbb H)= ker \ {}^\theta\mathcal D_r$ are called left (respectively right) $\theta$-hyperholomorphic functions on $\Omega$.

The spaces ${}^\theta\mathfrak M$ and ${}^\theta \mathfrak M_r$ are right-quaternionic Banach modules, although they are also real linear spaces (from both sides). This has the disadvantage that ${}^\theta\mathcal D_r$ acts on them only as a real linear operator, not a quaternionic linear operator, but in our context it will be good enough since the principal operator under consideration is ${}^\theta\mathcal D$.
\\
In addition,  $f=f_1+f_2j\in {}^\theta\mathfrak M(\Omega, \mathbb H)$, where $f_1,f_2: \Omega\to \mathbb C(i)$, if and only if 
\begin{align}\label{equa111}
\left\{
\begin{array}{l}
\displaystyle{\frac{\partial f_1}{\partial \bar z_1}  =  
ie^{i\theta}\overline{\frac{\partial f_2
}{\partial \bar z_2}}}, \\
 \\
\displaystyle{\frac{\partial f_1}{\partial \bar z_2}  =  
-ie^{i\theta}\overline{\frac{\partial f_2}{\partial \bar z_1}}}.    
\end{array} \right.
\end{align}
Let us consider the following spaces of holomorphic maps:
\[
Hol(\Omega,\mathbb C(i))=\{f\in C^{1}(\Omega, \mathbb C(i)): \displaystyle\frac{\partial f}{\partial \bar z_1}=0, \displaystyle\frac{\partial f}{\partial \bar z_2}=0\}
\]
and more generally
\[
Hol(\Omega,\mathbb C(i)^2)=\{f=(f_1, f_2): f_1, f_2 \in Hol(\Omega,\mathbb C(i))\}.
\]
The relation
\[\displaystyle{Hol(\Omega,\mathbb C(i)^2)=\bigcap_{0\leq \theta < 2\pi} {}^\theta\mathfrak M(\Omega)}\]
can be found in \cite{MS}.

Given $q, \xi \in \mathbb H$, we define 
$$\langle q,\xi \rangle_{\theta}=\sum_{k=0}^3 x_k y_k.$$ 

Let us introduce the temporary notation $q_{\theta}$ and $\xi_{\theta}$ for $q$ and $\xi$.
$$q_{\theta}=x_0 +x_1 i +x_2 i e^{i\theta} j+ x_3 e^{i\theta} j= z_1+ ie^{i\theta}j z_2 $$  
and 
$$\xi_{\theta} =y_0 +y_1 i +y_2 i e^{i\theta} j+ y_3 e^{i\theta}j=\zeta_1+ ie^{i\theta}j \zeta_2,$$ 
where $z_1=x_0+ix_1, z_2=x_2+ix_3,\zeta_1=y_0+iy_1, \zeta_2=y_2+iy_3$. 

The mappings $q_{\theta}\to (z_1,z_2)$ and $\xi_{\theta}\to (\zeta_1,\zeta_2)$ establish the following  operations in $\mathbb C(i)^2$, according to the quaternionic algebraic structure.  
\begin{itemize}
\item $(z_1, z_2)  \pm (\zeta_1 , \zeta_2)=  (z_1 \pm   \zeta_1, z_2\pm \zeta_2),$ 
\item $(z_1, z_2)  (\zeta _1, \zeta _2)=(z_1 \zeta _1- \bar{z}_2  \zeta _2,  \bar z_1 \zeta_2+ z_2 \zeta _1).$
\item $\overline{ (z_1, z_2)}=(\bar{z}_1 ,-z_2)$.
\item $|z_1+ z_2j|^2:=|z_1|_{\mathbb C(i)}^2+|z_2|_{\mathbb C(i)}^2=(z_1, z_2)(\bar z_1 , -z_2)$. 
\item Topology in $\mathbb C(i)^2$ is determined by the metric $d(q_{\theta}, \xi_{\theta}):= |q_{\theta}-\xi_{\theta}|$ for $q_{\theta},\xi_{\theta}\in\mathbb C(i)^2$.
\end{itemize}
Unless otherwise stated we continue to write $q$ for $q_{\theta}$ and $\xi$ for $\xi_{\theta}$.

Let $\Omega\subset \mathbb C(i)^2$ be a bounded domain with its boundary $\partial \Omega$ a compact $3-$dimensional sufficiently smooth hypersurface (co-dimension 1 manifold). The following formulas can be found in many sources (see, for example \cite{GS, MS}).

Given $f, g \in C^1(\overline{\Omega},\mathbb H)$ we have the quaternionic Stokes formulas
\begin{equation} \label{Dif-stokes} d(g\sigma^{\theta}_{\xi}  f) = 
  g \ {}^{{     \theta   }}\mathcal  D[f]+ \  \mathcal D^{{     \theta   }}[g] f ,
	\end{equation}
	\begin{equation}	 \label{Int-stokes}   \int_{\partial \Omega} g
	\sigma^{\theta}_{\xi}
	f     =       \int_{\Omega } \left( g    {}^\theta    \mathcal  D[f] +   \mathcal  D^{{     \theta   }}[g] f   \right)  d\mu, 
\end{equation}
and the quaternionic Borel-Pompieu formula  
\begin{align}  &  \int_{\partial \Omega }  ( K_{\theta   }(\xi - q )\sigma_{\xi
}^{\theta} f(\xi )  +  g(\xi )   \sigma_{\xi }^{\theta   } K_{\theta   }(\xi - q  ) ) \nonumber  \\ 
&  - 
\int_{\Omega} (K_{\theta   } ( \xi- q )
  {}^{\theta   }\mathcal D [f] ( \xi )  +    \mathcal  D^{{     \theta   }} [g] ( \xi)   K_{\theta   } ( \xi - q )
     )d\mu_{\xi}   \nonumber \\
		=  &    \label{ecua4}   \left\{ \begin{array}{ll}  f( q ) + g( q ) , &   q \in \Omega,  \\ 0 , &   q \in \mathbb H\setminus\overline{\Omega},    \end{array} \right. 
\end{align} 
where 
$$\displaystyle{ K_{\theta}(\xi -q ):= \frac{ 1}{2\pi^2
|\xi- q|^4} [(\bar\zeta_1- \bar z_1) - i e^{-i\theta}
(\bar\zeta_2-  \bar z_2)j ]}$$
is called $\theta$-hyperholomorphic Cauchy kernel and 
$$\sigma^{\theta}_{\xi}:=\frac{1}{2} [d \bar \xi_{[2]} \wedge d\xi - ie^{i\theta}d \xi_{[1]} \wedge d\bar \xi j],$$
where $d \xi_{[i]}$ denotes as usual $d y_0 \wedge d y_1 \wedge dy_2 \wedge d y_3$ omitting the factor $d y_i$, represents the $\mathbb H-$valued area form to $\partial \Omega$ and $d\mu_\xi$ stands for the $4-$dimensional volume element in $\Omega$. Here  $C^1(\overline{\Omega},\mathbb H)$ denotes the subclass of functions which can be extended smoothly to an open set containing the closure $\overline{\Omega}$.

From the above we can see that the $\theta$-hyperholomorphic Cauchy kernel generates the following important operator.
\begin{equation}\label{e3}
 {}^\theta T [f](q):=\int_{\Omega} K_\theta(\xi-q) f(\xi) d\mu_\xi
\end{equation}
on $\mathcal L_2(\Omega,\mathbb H)\cup C(\Omega,\mathbb H)$ and meets ${}^\theta\mathcal D  \circ  {}^\theta T  = I$. This can be found in \cite{GM, MS, SV1, SV2}.

The elements of $\mathbb H$ are written in terms of the structural set $\psi$ hence those of $\mathbb H(\mathbb C(\textsf{i}))$ can be written as $q=\sum_{k=0 }^3 \psi_k q_k,$ where $q_k\in \mathbb C(\textsf{i})$.
 
\section{Main results}
We shall consider vector parameters $\vec{\alpha} := (\alpha_0, \alpha_1,\alpha_2,\alpha_3); \vec{\beta} := (\beta_0, \beta_1,\beta_2,\beta_3) \in\mathbb C(\textsf{i})^4$, under the condition that $0<\Re\alpha_{\ell}; \Re\beta_{\ell}<1$, for $\ell=0,1,2,3$.
\subsection{Fractional Cauchy-Riemann system of order $\vec{\alpha}$}
\begin{definition}
Given  $a=(a_1,a_2)$, $b=( 
b_1,b_2)  \in \mathbb C(i)^2$ such that $\Re a_k< \Re  b_k$ and  $\Im a_k< \Im  b_k$  for  $k=1,2$, denote the rectangle   
\begin{align*}  {J_a^b }:= &  \{  (z_1, z_2) \in \mathbb C(i)^2 \ \mid \ \Re a_k< \Re z_k < \Re b_k, \  
 \Im a_k< \Im z_k < \Im b_k , \ k=0,1,2,3\} \\
 = & (\Re a_1,  \Re b_1 ) \times (\Im a_1,  \Im b_1 ) \times (\Re a_2,  \Re b_2 ) \times (\Im a_2,  \Im b_2 ) ,
\end{align*}
and define $m(J_a^b):=( \Re b_1 -\Re a_1) ( \Im b_1-\Im a_1)( \Re b_2-\Re a_2)(\Im b_2-\Im a_2)$. 
\end{definition}

\begin{remark} 
Let $\vec{\alpha} = (\alpha_0, \alpha_1,\alpha_2,\alpha_3) \in\mathbb C(\textsf{i})^4$ such that $0< \Re \alpha_{\ell}<1$ for $\ell=0,1,2,3$  and $f=f_1 + f_2 j \in AC^1(J_a^b,\mathbb H)$; i.e.,  $f_1,f_2 \in AC^1(J_a^b,\mathbb C(i))$, or equivalently, the real and the imaginary components of $f_1$ and $f_2$  belong to $AC^1(J_a^b,\mathbb R)$.

Given $q=(z_1, z_2)\in J_a^b$ and $f= f_1 + f_2 j \in AC^1(J_a^b,\mathbb H)$, we can certainly write $q\cong q_{\theta}=z_1+ie^{i\theta}j z_2$ and     
$f_n (q)\cong f_n (z_1, z_2)$ for $n=1,2$. Note that this is in agreement with our previously introduced terminology.

In order to define fractional derivatives in the Riemann-Liouville sense of the mappings 
\begin{align*}
\Re z_1   \mapsto \Re f_1 ( \Re z_1 + i \Im \zeta_1 , \zeta_2 ), \ \ 
\Im z_1 \mapsto \Im f_1( \Re \zeta_1 + \Im z_1, \zeta_2 ), \\ 
\Re z_2   \mapsto \Re f_2 ( \zeta_1, \Re z_2 + i \Im \zeta_2 ), \ \ 
\Im z_2 \mapsto \Im f_2(  \zeta_1, \Re \zeta_2 + \Im z_2 ), 
\end{align*}
we assume $q,\xi\in J_a^b$, where $\xi$ is a fix. Indeed, they are given by the following $\mathbb C(\textsf{i})$-valued functions:
\begin{align*}
\Re z_1   \mapsto D_{\Re {a_1}^+}^{\alpha_0}\Re f_1 ( \Re z_1 + i \Im \zeta_1 , \zeta_2 ), \ \ 
\Im z_1 \mapsto  D_{\Im {a_1}^+}^{\alpha_1} \Im f_1( \Re \zeta_1 + \Im z_1, \zeta_2 ) ,\\ 
\Re z_2   \mapsto D_{\Re {a_2}^+}^{\alpha_2}\Re f_2 ( \zeta_1, \Re z_2 + i \Im \zeta_2 ), \ \ 
\Im z_2 \mapsto D_{\Im {a_2}^+}^{\alpha_3} \Im f_2(  \zeta_1, \Re \zeta_2 + \Im z_2 ), 
\end{align*}
respectively.
	
Note that $D_{\Re {a_1}^+}^{\alpha_0}$ is applied in variable $\Re z_1$,  $D_{\Im {a_1}^+}^{\alpha_1}$ is applied in $\Im z_1$ and the same dependency  for the rest of the derivatives. 
\end{remark}

\begin{definition} Give  $\vec{\alpha} = (\alpha_0, \alpha_1,\alpha_2,\alpha_3) \in\mathbb C(\textsf{i})^4$ such that $0< \Re \alpha_{\ell}<1$ for $\ell=0,1,2,3$  and $g \in AC^1(J_a^b,\mathbb H)$, where $a=(a_1,a_2), b=(b_1,b_2) \in \mathbb C(i)^2$. The fractional $\theta$-Fueter operator of order $\vec{\alpha}$ is defined to be
\begin{align*} 
{}^{\theta}\mathfrak D_a^{\vec{\alpha}}[g] (\xi, q):= & ( D _{\Re a_1^+}^{{\alpha_0}}g)(
 \Re z_1 + i \Im \zeta_1 , \zeta_2 )  
    +  i   ( D _{\Im a_1^+}^{{\alpha_1}}g) ( \Re \zeta_1 + \Im z_1, \zeta_2 )
    \\
 &		+   i e^{i\theta}j    ( D _{\Re a_2^+}^{{\alpha_2}}g)( \zeta_1, \Re z_2 + i \Im \zeta_2 )         +   e^{i\theta}j     ( D _{\Im a_2^+}^{{\alpha_3}}g)( \zeta_1, \Re \zeta_2 + \Im z_2 ) ,
			\end{align*}
for $q, \xi \in J_a^b$. Note that $\xi$ is considered a fixed point since the integration and derivation variables are the real components of $q$. Moreover, the mapping 
$$q\mapsto {}^{\theta}\mathfrak D_a^{\vec{\alpha}}[f]( \xi, q )$$ 
is a $\mathbb H(\mathbb C(\textsf{i}))$-valued function.

The $\mathbb C(i)((\mathbb C(\textsf{i} ))$-components of the previous operators, denoted by  
\begin{align}\label{ComplexOper}  \mathfrak D_{a_1}^{(\alpha_0,\alpha_1)}[g](\xi,q) := &   D_{\Re {a_1}^+}^{\alpha_0} g  ( \Re z_1 + i \Im \zeta_1 , \zeta_2 ) + i D_{\Im {a}_1^+}^{\alpha_1}  g ( \Re \zeta_1 + \Im z_1, \zeta_2 ) , \nonumber\\
  \mathfrak D_{a_2}^{(\alpha_2,\alpha_3)}[g](\xi,q):= &   D_{\Re {a}_2^+}^{\alpha_2} g  ( \zeta_1, \Re z_2 + i \Im \zeta_2 ) + i  D_{\Im {a_2}^+}^{\alpha_3}  g (  \zeta_1, \Re \zeta_2 + \Im z_2), 
\end{align}
give
  \begin{align*} 
{}^{\theta}\mathfrak D_a^{\vec{\alpha}}[g] (\xi, q) =    \mathfrak D_{a_1}^{(\alpha_0,\alpha_1)}[g](\xi,q) 
 + i e^{i\theta}j    \mathfrak D_{a_2}^{(\alpha_2,\alpha_3)}[g](\xi,q).  
\end{align*}

Note that $\mathfrak D_{a_1}^{(\alpha_0,\alpha_1)}[g](\xi, \cdot)$ and $\mathfrak D_{a_2}^{(\alpha_2,\alpha_3)}[g](\xi, \cdot)$ are $\mathbb H( \mathbb C(\textsf{i}))$-valued functions. Particularly, for $g \in AC^1(J_a^b,\mathbb C(i))$ we will consider: 
\begin{align*} 
		\mathfrak I_{a_1}^{(\alpha_0, \alpha_1)}[g](\xi, q) := &  \frac{1}{\Gamma (\alpha_0)}
\int_{\Re a_1}^{\Re z_1}\frac{
\Re g  (\tau_0 + \Im \zeta_1, \zeta_2 )}{ ( \Re z_1 - \tau_0)^{1-\alpha_0}} 
d\tau_0    + \frac{1}{\Gamma (\alpha_1)}
\int_{\Im a_1}^{\Im z_1}\frac{
\Im g  ( \Re \zeta_1 + i \tau_1 , \zeta_2 )}{ ( \Im z_1 - \tau_1)^{1-\alpha_1}} 
d\tau_1  ,\\
 \mathfrak I_{a_2}^{(\alpha_2, \alpha_3)}[g](\xi, q) := & \frac{1}{\Gamma (\alpha_2)}
\int_{\Re a_2}^{\Re z_2}\frac{
-\sin( \theta) \Re g  (  \zeta_1 ,  \tau_2+ i \Im \zeta_2 ) + \cos (\theta ) \Im g  (  \zeta_1 ,  \tau_2+ i \Im \zeta_2 )  }{ ( \Re z_2 - \tau_2)^{1-\alpha_2}} 
d\tau_2    \\
 & + \frac{1}{\Gamma (\alpha_3)}
\int_{\Im a_2}^{\Im z_2}\frac{
\sin (\theta) \Im g  (   \zeta_1, \Re \zeta_2 + i \tau_3 )+ \cos (\theta) \Re g  (   \zeta_1, \Re \zeta_2 + i \tau_3 )}{ ( \Im z_2 - \tau_3)^{1-\alpha_3}} 
d\tau_3 . 
\end{align*}
The functions $\mathfrak I_{a_1}^{(\alpha_0, \alpha_1)}[g](\xi, \cdot)$ and $\mathfrak I_{a_2}^{(\alpha_2, \alpha_3)}[g](\xi, \cdot)$  are  $\mathbb C(i)( \mathbb C(\textsf{i}))$-valued. 

Also, for every $f=f_1+f_2 j \in AC^1(J_a^b,\mathbb H)$, i.e.,  $f_1,f_2 \in AC^1(J_a^b,\mathbb C(i))$, we define 
\begin{align*}
\mathfrak I_a^{\vec{\alpha}}[f](\xi, q) := & \mathfrak I_a^{(\alpha_0, \alpha_1)}[f_1](\xi, q) + 
\mathfrak I_a^{(\alpha_2, \alpha_3)}[f_2](\xi, q)j, \nonumber 
\end{align*}
and
\begin{align*}
\mathcal I_a^q [f] (\xi,q,\vec{\alpha})    
	 := &  \int_{J_a^q } \frac{1}{m(J_a^q)}  \left(\frac{f  
(\Re\tau_1+ i\Im \zeta_1, \zeta_2)  ( \Re (z_1-\tau_1) )^{ \alpha_0}}{\Gamma(\alpha_0) } \right. \nonumber \\
 & + \frac{f  (\Re\zeta_1+ i\Im \tau_1, \zeta_2)  ( \Im (z_1-\tau_1))^{ \alpha_1}}{\Gamma(\alpha_1) }   +  \frac{f  ( \zeta_1,  \Re \tau_2+ i \Im \zeta_2)  ( \Re (z_2-\tau_2))^{ \alpha_2}}{\Gamma(\alpha_2) } \nonumber\\
  & \left. +  \frac{f  ( \zeta_1,  \Re \zeta_2+ i \Im \tau_2)  ( \Im (z_2-\tau_2))^{ \alpha_3}}{\Gamma(\alpha_3) } \right)   d\mu_\tau,\nonumber  
\end{align*}
where  $\tau = (\tau_1, \tau_2)\in\mathbb C(i)^2$ and $d\mu_\tau$ is the differential  of four dimensional volume.
\end{definition}

\begin{remark}\label{remark001}
For every $f=f_1+f_2 j \in AC^1(J_a^b,\mathbb H)$, i.e.,  $f_1,f_2 \in AC^1(J_a^b,\mathbb C(i))$, we can see that 
\begin{align} \label{thetafracOper}
    \mathfrak D_{a}^{\vec \alpha}[f](\xi,q)  =  &  \left(	  
  \mathfrak D_{a_1}^{(\alpha_0,\alpha_1)}[f_1] (\xi,q)  -i e^{i\theta}  \overline{   \mathfrak D_{a_2}^
	{(\alpha_2,\alpha_3)}  [f_2]   (\xi,q)} \right)  \nonumber \\ 
 	 &   +  \left(
	\mathfrak D_{a_1}^{(\alpha_0,\alpha_1)} 
	[f_2 ] (\xi,q)
	+ i e^{i\theta}  \overline{  \mathfrak D_{a_2}^{(\alpha_2,\alpha_3)}
	  [f_1](\xi,q)}\right)j.  
\end{align}

Note that, mappings $q\mapsto {}^{\theta}\mathfrak D_a^{\vec{\alpha}}[f](\xi,q)$,  $q\mapsto \mathcal I_a^q [f] (\xi,q ,\vec{\alpha})$ and  $q\mapsto \mathfrak I_a^{\vec{\alpha}}[f](\xi, q)$ are $\mathbb H(\mathbb C(\textsf{i}))$-valued function.

Particularly, ${}^{\theta}\mathfrak D_a^{\vec{\alpha}}[f] (\xi,q) \mid_{q=\xi}$ can be considered as the fractional derivative of order $\vec \alpha$ at the point $\xi\in J_a^ b$. We shall write ${}^{\theta}\mathfrak D_a^{\vec{\alpha}}[f] (\xi)$ instead of ${}^{\theta}\mathfrak D_a^{\vec{\alpha}}[f] (\xi,q) \mid_{q=\xi}$.

Note that $\xi$ is considered fixed since the integration and derivation variables are the real components of $z_1 $ and $z_2$.  

The fractional differential equation ${}^{\theta}\mathfrak D_a^{\vec{\alpha}}[f(\xi,\cdot)]=0$ on $J_a^b$ is equivalent to the fractional Cauchy Riemann system 
\begin{align}\label{FracC-R-Equa} 
 \mathfrak D_{a_1}^{(\alpha_0,\alpha_1)}[f_1] (\xi,q)  =  & i e^{i\theta}  \overline{   \mathfrak D_{a_2}^{(\alpha_0,\alpha_1)}  [f_2]   (\xi,q)} , \nonumber \\
\mathfrak D_{a_2}^{(\alpha_2,\alpha_3)}
	  [f_1](\xi,q)
=& - i e^{i\theta}
 \overline{ 	\mathfrak D_{a_1}^{(\alpha_0,\alpha_1)} 
	[f_2 ] (\xi,q)}, 
\end{align} 
for every $q\in J_a^b$, which can be seen as a well analogous of system \eqref{equa111}. 
\end{remark}
We are interested in studying null-solutions of the operator ${}^{\theta}\mathfrak D_a^{\vec{\alpha}}$ and the fractional Riemann-Liouville integrals $\mathfrak I_a^{\vec{\alpha}}$ and $\mathcal I_a^q$. 

\begin{proposition}\label{propFRACD}
Let  $f=f_1+f_2 j \in AC^1(J_a^b,\mathbb H)$, $ \vec{\alpha} = (\alpha_0, \alpha_1,\alpha_2,\alpha_3) \in\mathbb C(\textsf{i})^4$ with $0< \Re \alpha_{\ell} <1$ for $\ell=0,1,2,3$ and set $q=(z_1,z_2) , \ \xi= (\zeta_1,\zeta_2) \in J_a^b$. The following identities hold:   
\begin{enumerate} 
\item ${}^{\theta}\mathfrak D_a^{\vec{\alpha}}[f](\xi,q) = {}^{\theta}  \mathcal D_{q} \circ  \mathcal I_a^q [f](\xi, q,\vec{\alpha}).$
\item 
\begin{align*} {}^{\theta}\mathfrak D_a^{\vec{\alpha}} \circ    \mathfrak I_a^{\vec{\alpha}}[f](\xi,q) = & \Re f_1( \Re z_1 + i \Im \zeta_1, \zeta_2) + i \Im f_1( \Re \zeta_1  + i \Im z_1, \zeta_2 ) \\ 
&+ ( \Re f_2(\zeta_1 , \Re z_2 + i \Im \zeta_2 ) + i \Im f_2(   \zeta_1 , \Re\zeta_2 + i \Im z_2))j. 
\end{align*}

Particularly, ${}^{\theta}\mathfrak D_a^{\vec{\alpha}} \circ  \mathfrak I_a^{\vec{\alpha}}[f](\xi,q)\mid_{q=\xi}= f(\xi)$. 

\item $ {}^{\bar \theta}\mathcal D_q\circ  {}^{\theta}\mathfrak D_a^{\vec{\alpha}}[f](\xi,q) =  (\Delta_{z_1} +  \Delta_{z_2}) \circ  \mathcal I_a^q [f](\xi, q,\vec{\alpha})$, 
where $\Delta_{z_1}$ and  $\Delta_{z_2} $ denotes the Laplacian operators in $\mathbb R^2$, corresponding to the real components of $z_1$ and $z_2$, respectively.
\item  If  the mapping $q\to \mathcal I_a^{q} [f](\xi, q,\vec{\alpha})$ belongs to $C^2(J_a^b, \mathbb H)$ for $\xi$ fix and set 
$\vec{\beta} = (\beta_0, \beta_1, \beta_2, \beta_3) \in\mathbb C(\textsf{i})^4$ with $0< \Re(\alpha_{\ell}+\beta_{\ell}) <1$ for $\ell=0,1,2,3$  then we have  
\begin{align*} {}^{\theta}\mathfrak D_a^{\vec{\alpha}} \circ  {}^{\theta}\mathfrak D_a^{\vec{\beta}}[f] (\xi,q) = &   D _{\Re a_1^+}^{{\alpha_0 + \beta_0}} f(\Re z_1 + i \Im \zeta_1 , \zeta_2) -   D _{\Im a_1^+}^{{\alpha_1 + \beta_1}} f(\Re \zeta_1 + i \Im z_1 , \zeta_2) \\
 & -  D _{\Re a_2^+}^{{\alpha_2 + \beta_2}} f( \zeta_1, \Re z_2 + i \Im \zeta_2 )-D _{\Im a_2^+}^{{\alpha_3 + \beta_3}} f(\zeta_1, \Re \zeta_2 + i\Im z_1)
\end{align*} 
and 
\begin{align*} \overline{ {}^{\theta}\mathfrak D_a^{\vec{\alpha}} } \circ  {}^{\theta}\mathfrak D_a^{\vec{\beta}}[f] (\xi,q) = &   D _{\Re a_1^+}^{{\alpha_0 + \beta_0}} f(\Re z_1 + i \Im \zeta_1 , \zeta_2) +  D _{\Im a_1^+}^{{\alpha_1 + \beta_1}} f(\Re \zeta_1 + i \Im z_1 , \zeta_2) \\
 & +  D _{\Re a_2^+}^{{\alpha_2 + \beta_2}} f( \zeta_1, \Re z_2 + i \Im \zeta_2)+D _{\Im a_2^+}^{{\alpha_3 + \beta_3}} f(\zeta_1, \Re \zeta_2 + i\Im z_1),
\end{align*}
where 
$$\overline{{}^{\theta} \mathfrak D_a^{\vec{\alpha}} } [f](\xi,q) = \overline{ {}^{  \theta}  \mathcal D_{q}}\circ  \mathcal I_a^q [f](\xi, q,\vec{\alpha}).$$
Note that, for $\vec{\alpha}=\vec{\beta}$, the above formula drawn the fact that the fractional  operator of order $\displaystyle\frac{1+\vec{\alpha}}{2}$ factorizes a fractional  Laplace operator: 
$${}^{\theta}\Delta_a^{\vec{\alpha}}:=\sum_{j=0}^3 D _{a_j^+}^{{1+\alpha_j}}.$$
\end{enumerate}
\end{proposition}
\begin{proof} 
\begin{enumerate}
\item Facts 3. and 4. are direct consequence of \cite [Proposition 3.3]{BG1} for the structural set $\{ 1,i, ie^{i \theta} j ,  e^{i \theta} j\}$. 
\item We may rewrite $f=  f_1+ f_2 j$ in terms of $\{ 1,i, ie^{i \theta} j ,  e^{i \theta} j  \}$ as 
\begin{align*}f= &f_1+ f_2 j =   \Re f_1 + i \Im f_1 +    e^{i\theta }j     e^{i\theta }\Re f_2 + i  e^{i\theta } j   e^{i\theta } \Im f_2 \\
=  &   \Re f_1 + i \Im f_1 +   i e^{i\theta }j    (-\sin (\theta )\Re f_2 + \cos(\theta) \Im f_2 ) +
   e^{i\theta }j    (\sin (\theta ) \Im f_2 + \cos(\theta)  \Re f_2  )
\end{align*}
to use Fact 2 of \cite[Proposition 3.3]{BG1}, which completes the proof. 
\end{enumerate}
\end{proof}

\begin{remark}\label{remark01}
Given $g  \in AC^1(J_a^b,\mathbb H)$  denote  
$${}^{\theta}\mathfrak D_{r,a}^{\vec{\beta}}[g](\xi,q) =  {}^{\theta}  \mathcal D_{q,r} 
\circ  \mathcal I_a^q [f](\xi, q,\vec{\beta}) $$
and note that 
\begin{align*} 
{}^{\theta}\mathfrak D_{r,a}^{\vec{\beta}}[g] (\xi, q):= & ( D _{\Re a_1^+}^{{\beta_0}}g)(\Re z_1 + i \Im \zeta_1 , \zeta_2)  
    +      ( D _{\Im a_1^+}^{{\beta_1}}g)(  \Re \zeta_1 + \Im z_1, \zeta_2) i  
    \\
 &		+     ( D _{\Re a_2^+}^{{\beta_2}}g)( \zeta_1, \Re z_2 + i \Im \zeta_2)  i e^{i\theta}j  
       +       ( D _{\Im a_2^+}^{{\beta_3}}g)(   \zeta_1, \Re \zeta_2 + \Im z_2 )  e^{i\theta}j.
\end{align*}

Thus, defining
\begin{align*}  \mathfrak D_{r,a_1}^{(\beta_0,\beta_1)}[g](\xi,q) = &   D_{\Re {a_1}^+}^{\beta_0} g  ( \Re z_1 + i \Im \zeta_1 , \zeta_2 ) +  D_{\Im {a}_1^+}^{\beta_1}  g ( \Re \zeta_1 + \Im z_1, \zeta_2 ) i , \nonumber\\
  \mathfrak D_{r,a_2}^{(\beta_2,\beta_3)}[g](\xi,q)= &   D_{\Re {a}_2^+}^{\beta_2} g  ( \zeta_1, \Re z_2 + i \Im \zeta_2 ) -  D_{\Im {a_2}^+}^{\beta_3}  g (  \zeta_1, \Re \zeta_2 + \Im z_2 ) 
	i, \end{align*}
we can see that
  \begin{align*} 
{}^{\theta}\mathfrak D_{r,a}^{\vec{\beta}}[g] (\xi, q) =    \mathfrak D_{r, a_1}^{(\beta_0,\beta_1)}[g](\xi,q) 
 +    \mathfrak D_{r, a_2}^{(\beta_2,\beta_3)}[g](\xi,q) i e^{i\theta}j.  
\end{align*}
\end{remark}

\begin{remark}\label{Remark1}
Properties exhibited by Proposition \ref{propFRACD} gives an extension of basic formulas related to the standard fractional Riemann-Louville derivative to the context of a fractional quaternionic analysis. This, essentially with 
$\dfrac{d}{dx}$ and  $({\bf I}_{a^+}^{\vec{\alpha}} f)(x)$ replaced by ${}^{\theta}\mathcal D$ and $\mathcal I_a^q [f](\xi, q,\vec{\alpha})$ respectively.

In particular, Fact 2. establish a quaternionic analogous of the Fundamental Theorem, comparing with \eqref{FundTheorem}.
\end{remark}
	
\begin{proposition}\label{Stokes}(Stokes type integral formula induced by ${}^{\theta} \mathfrak D_{a}^{\vec{\alpha}}$)
 
\noindent 
Let $\vec{\alpha}= (\alpha_0, \alpha_1,\alpha_2,\alpha_3), \vec{\beta}= (\beta_0, \beta_1,\beta_2,\beta_3) \in\mathbb C(\textsf{i})^4$ with 
  $0< \Re\alpha_{\ell}, \Re\beta_{\ell}<1$ for $\ell=0,1,2,3$ and given $f,g \in AC^1(\overline{J_a^b}, \mathbb H)$ with $\xi\in J_a^b$ such that   
 the mappings $q\mapsto  \mathcal I_a^q [f](\xi,q, \vec{\alpha})$ and $ q\mapsto \mathcal I_a^q [g](\xi,q, \vec{\beta} )$ belong to  $ C^1(\overline{J_a^b}, \mathbb H(\mathbb C(\textsf{i})))$.   
 Then 
 \begin{align*} &   \int_{\partial J_a^b}   \mathcal I_a^q [g](\xi,q, \vec{\beta}) \sigma^{{\theta} }_q  \mathcal I_a^q [f](\xi,q, \vec{\alpha}) \\ 
=  &      \int_{J_a^b }  \left(   \mathcal I_a^q [g](\xi,q,\vec{\beta}) \  {}^{\theta}\mathfrak D_a^{\vec{\alpha}}[f](\xi,q) + \  {}^{\theta} \mathfrak D_{r,a}^{\vec{\beta}}[g](\xi,q)  \mathcal I_a^q [f](\xi,q,\vec{\alpha})\right)d\mu_q.
\end{align*}
 \end{proposition}
\begin{proof} Use \eqref{Int-stokes} and the first identity of Proposition \ref{propFRACD}. It follows by similar reasoning to \cite[Proposition 3.4]{BG1}.  
\end{proof}
  
\begin{theorem}\label{B-P-F-D}(Borel-Pompieu type formula induced by ${}^{\theta}\mathfrak D_{a}^{\vec{\alpha}}$ and $ {}^{\theta}\mathfrak D_{r,a}^{\vec{\beta}}$)
 
\noindent  
Let $\vec{\alpha}= (\alpha_0, \alpha_1,\alpha_2,\alpha_3), \vec{\beta}= (\beta_0, \beta_1,\beta_2,\beta_3) \in\mathbb C(\textsf{i})^4$ with $0< \Re\alpha_{\ell}, \Re\beta_{\ell}<1$ for $\ell=0,1,2,3$ and $f,g \in AC^1(\overline{J_a^b}, \mathbb H)$. Consider $\xi\in J_a^b$ such that the mappings   
 $q\to  \mathcal I_a^q [f](\xi,q,\vec{\alpha})$ and $q\to   \mathcal I_a^q [g](\xi,q, \vec{\beta})$, for $q\in J_a^b$, belong to $C^1(\overline{J_a^b}, \mathbb H(\mathbb C(\textsf{i})))$. Then 
\begin{align*}  &  
  \int_{\partial J_a^b } \left(\mathfrak K^{\vec{\alpha}}_{\theta, a}(\tau,q) \sigma_{\tau}^{\theta} \mathcal I_a^{\tau} [f](\xi,\tau, \vec{\alpha})  
		+     \mathcal I_a^{\tau} [g](\xi,\tau,\vec{\beta})  \sigma_{\tau}^{\theta} \mathfrak K^{\vec{\beta}}_{\theta, a}(\tau,q) \right) \\
		&		- 
\int_{J_a^b} \left(  \mathfrak K^{\vec{\alpha}}_{\theta, a}(\tau,q)  
   {}^{\theta}\mathfrak D_a^{\vec{\alpha}}[f](\xi , \tau)   +   
	   {}^{\theta} \mathfrak D_{r,a}^{\vec{\beta}}[g]( \xi, \tau )	\mathfrak K^{\vec{\beta}}_{\theta, a}(\tau, q)  
    		\right)  d\mu_{\tau}       \\
		=  &    \left\{ \begin{array}{ll}  (f+g) (\Re  z_1 + i\Im\zeta_1, \zeta_2) + 
		(f+g) (\Re  \zeta_1 + i\Im z_1, \zeta_2)  & \\
		+ 
		(f+g) (\zeta_1, \Re  z_2 + i\Im\zeta_2) + (f+g) (\zeta_1, \Re  \zeta_2 + i\Im z_2) & \\
+ N[f](\xi,q, \vec{\alpha}) + N[g](\xi,q, \vec{\beta}) , &    q\in 
		J_a^b ,  \\ 0 , &  q\in \mathbb C(i)^2\setminus\overline{J_a^b},                    
	\end{array} \right. 
	\end{align*} 
where  
\begin{align*} 
\mathfrak K^{\vec{\alpha}}_{\theta, a}(\xi,q) :=&     \frac{1}{2\pi^2} \left[ D _{{\Re a_1^+, \Re z_1} }^{\alpha_0}    \frac{ \bar\xi_1- \bar z_1}{ 
|\xi- q|^4}   + D_{{\Im a_1}^+, \Im z_1 }^{\alpha_1}    \frac{ \bar\xi_1- \bar z_1}{ 
|\xi- q|^4}  \right.  \\  
 & - i e^{-i\theta} \left. \left( D _{{\Re a_2^+, \Re z_2} }^{\alpha_2}   \frac{ 
 \bar\xi_2-  \bar z_2  }{ 
|\xi- q|^4} 
 + D _{{\Im a_2^+, \Im z_2} }^{\alpha_3}   \frac{ 
 \bar\xi_2-  \bar z_2  }{ 
|\xi- q|^4}  \right)j \right].
\end{align*}
Hence
$$\mathfrak K^{\vec{\alpha}}_{\theta, a}(\xi,q) = K^{ \alpha_0, \alpha_1}_{a}(\xi,q) +  K^{ \alpha_2, \alpha_3}_{a}(\xi,q) j,$$
where
\begin{align}\label{kernels} 
K^{ \alpha_0, \alpha_1}_{a}(\xi,q) := &   \frac{1}{2\pi^2} \left( D _{{\Re a_1^+, \Re z_1} }^{\alpha_0}    \frac{ \bar\xi_1- \bar z_1}{ 
|\xi- q|^4}   + D_{{\Im a_1}^+, \Im z_1 }^{\alpha_1}    \frac{ \bar\xi_1- \bar z_1}{ 
|\xi- q|^4} \right) \nonumber \\ 
 K^{ \alpha_2, \alpha_3}_{a}(\xi,q) := &
 -  \frac{1}{2\pi^2} i e^{-i\theta} \left( D _{{\Re a_2^+, \Re z_2} }^{\alpha_2}   \frac{ 
 \bar\xi_2-  \bar z_2  }{ 
|\xi- q|^4} + D _{{\Im a_2^+, \Im z_2} }^{\alpha_3}\frac{\bar\xi_2-  \bar z_2  }{|\xi- q|^4}\right). 
\end{align}
				
Note that	$D _{{\Re a_1, \Re z_1}^+}^{\alpha_0}$ is given in terms of $\Re z_1$ and the same behavior for the rest of derivatives. Moreover,    		  
  \begin{align*} & N[f ](\xi,q, \vec{\alpha}) = \\
	 & 
	({\bf I}_{\Re a_1^+}^{\alpha_0} f  ) 
(\Re z_1 + i \Im \zeta_1, \zeta_2) 	 \left(\frac{1} { \Gamma[\alpha_1] (\Im(z_1-a_1))^{\alpha_1}} +\frac{1} { \Gamma[\alpha_2] (\Re(z_2-a_2))^{\alpha_2}} + \frac{1} { \Gamma[\alpha_3] (\Im(z_2-a_2))^{\alpha_3}}
\right) \\
&  + 
({\bf I}_{\Im a_1^+}^{\alpha_1} f  ) 
(\Re \zeta_1 + i \Im z_1, \zeta_2) 	 \left(\frac{1} { \Gamma[\alpha_0] (\Re(z_1-a_1))^{\alpha_0}} +\frac{1} { \Gamma[\alpha_2] (\Re(z_2-a_2))^{\alpha_2}} + \frac{1} { \Gamma[\alpha_3] (\Im(z_2-a_2))^{\alpha_3}}
\right) \\
& +
	({\bf I}_{\Re a_2^+}^{\alpha_2} f  ) 
(\zeta_1, \Re z_2 + i \Im \zeta_2) 	 \left(\frac{1} { \Gamma[\alpha_0] (\Re(z_1-a_1))^{\alpha_0}} +\frac{1} { \Gamma[\alpha_1] (\Im(z_1-a_1))^{\alpha_1}} + \frac{1} { \Gamma[\alpha_3] (\Im(z_2-a_2))^{\alpha_3}}
\right) \\
& +
	({\bf I}_{\Im a_2^+}^{\alpha_3} f  ) 
(\zeta_1, \Re \zeta_2 + i \Im z_2) 	 \left(  \frac{1} { \Gamma[\alpha_0] (\Re(z_1-a_1))^{\alpha_0}} +\frac{1} { \Gamma[\alpha_1] (\Im(z_1-a_1))^{\alpha_1}} +\frac{1} { \Gamma[\alpha_2] (\Re(z_2-a_2))^{\alpha_2}} 
\right)
\end{align*}
and the same for $N[g](\xi,q, \vec{\beta})$.
\end{theorem} 
\begin{proof}
It is sufficient to use \cite[Theorem 3.5]{BG1} together with \cite[Remark 4]{BG1}.
\end{proof}

\begin{corollary} 
Let $\vec{\alpha}= (\alpha_0, \alpha_1,\alpha_2,\alpha_3), \vec{\beta}= (\beta_0, \beta_1,\beta_2,\beta_3) \in\mathbb C(\textsf{i})^4$ with $0< \Re\alpha_{\ell}, \Re\beta_{\ell}<1$ for $\ell=0,1,2,3$ and $f,g \in AC^1(\overline{J_a^b}, \mathbb H)$. Consider $\xi\in J_a^b$ such that the mappings   
 $q\to  \mathcal I_a^q [f](\xi,q,\vec{\alpha})$ and $q\to   \mathcal I_a^q [g](\xi,q, \vec{\beta})$, for all $q\in J_a^b$, belong to $C^1(\overline{J_a^b}, \mathbb H(\mathbb C(\textsf{i}))).$ Then 
\begin{align*}  & \int_{\partial J_a^b } \left(\mathfrak K^{\vec{\alpha}}_{\theta, a}(\tau,\xi) \sigma_{\tau}^{\theta} \mathcal I_a^{\tau} [f]( \xi ,\tau, \vec{\alpha})  
		+     \mathcal I_a^{\tau} [g]( \xi ,\tau ,\vec{\beta})  \sigma_{\tau}^{\theta} \mathfrak K^{\vec{\beta}}_{\theta, a}(\tau,\xi) \right) \\
		&		- 
\int_{J_a^b} \left(  \mathfrak K^{\vec{\alpha}}_{\theta, a}(\tau,\xi)  
  {}^{\theta} \mathfrak D_a^{\vec{\alpha}}[f]( \xi  , \tau)    +   
	  {}^{\theta}  \mathfrak D_{r,a}^{\vec{\beta}}[g](\xi, \tau )	\mathfrak K^{\vec{\beta}}_{\theta, a}(,\tau ,\xi)  
    		\right)  d\mu_{\xi}   \\
						= &   4 (f+g) (\zeta_1 , \zeta_2) 
+ N[f](\zeta,\zeta, \vec{\alpha}) + N[g](\zeta,\zeta, \vec{\beta}).                   
\end{align*} 
If $ {}^{\theta} \mathfrak D_a^{\vec{\alpha}}[f](\xi,\cdot) = {}^{\theta} \mathfrak D_{a,r}^{\vec{\alpha}}[g](\xi,\cdot) =0$ on $J_a^b$, we have
\begin{align*}  &  
  \int_{\partial J_a^b } \left(\mathfrak K^{\vec{\alpha}}_{\theta, a}(\tau,\xi) \sigma_{\tau}^{\theta} \mathcal I_a^{\tau} [f]( \xi ,\tau, \vec{\alpha})  
		+     \mathcal I_a^{\tau} [g]( \xi ,\tau ,\vec{\beta})  \sigma_{\tau}^{\theta} \mathfrak K^{\vec{\beta}}_{\theta, a}(\tau,\xi) \right) \\
						= &   4 (f+g) (\zeta_1 , \zeta_2) 
+ N[f](\zeta,\zeta, \vec{\alpha}) + N[g](\zeta,\zeta, \vec{\beta}).                   
\end{align*}  	 
\end{corollary} 

\subsection{Fractional Borel-Pompeiu type formula for holomorphic functions in two complex variables}\label{3.2}
Let $f \in AC^1(J_a^b,\mathbb C(i))$ and $\vec{\alpha} = (\alpha_0, \alpha_1,\alpha_2,\alpha_3) \in\mathbb C(\textsf{i})^4$ with $0< \Re \alpha_{\ell}<1$ for $\ell=0,1,2,3$. Set $a=(a_1,a_2), b=(b_1,b_2) \in \mathbb C(i)^2$ and    $q=(z_1, z_2),\ \xi=(\zeta_1, \zeta_2) \in J_a^{b}$. 
 
Using the operators given in  \eqref{ComplexOper} define  
\begin{align*}  
    {\bf D}_{a}^{\vec \alpha}[f](\xi,q) :=  &   	  
  \left( \mathfrak D_{a_1}^{(\alpha_0,\alpha_1)}[f ] (\xi,q)  , \   \mathfrak D_{a_2}^{(\alpha_2,\alpha_3)}
	  [f](\xi,q)  \right). 
\end{align*}
	
Particularly, the fractional derivative of order $\vec \alpha$ at the point $\xi\in J_a^ b$ is 
$$\left( \mathfrak D_{a_1}^{(\alpha_0,\alpha_1)}[f ] (\xi,\xi)  , \   \mathfrak D_{a_2}^{(\alpha_2,\alpha_3)} [f](\xi,\xi)  \right).$$

Moreover, the fractional differential equation 
$${\bf D}_a^{\vec{\alpha}}[f(\xi,\cdot)]=(0,0)$$ 
on $J_a^b$ is equivalent to the fractional Cauchy-Riemann system that arises from the following identities: 
\begin{align*}  
D_{\Re {a_1}^+}^{\alpha_0} f(\Re z_1 + i \Im \zeta_1 , \zeta_2) = & - i D_{\Im {a}_1^+}^{\alpha_1}  f ( \Re \zeta_1 + \Im z_1, \zeta_2 ) ,  \\
    D_{\Re {a}_2^+}^{\alpha_2} f(\zeta_1, \Re z_2 + i \Im \zeta_2) =  &- i  D_{\Im {a_2}^+}^{\alpha_3} f(\zeta_1, \Re \zeta_2 + \Im z_2), 
\end{align*} 
for every $q=(z_1,z_2)\in J_a^b$.

\begin{proposition}\label{propFRACD1}
Given  $f \in AC^1(J_a^b,\mathbb C(i))$  and $ \vec{\alpha} = (\alpha_0, \alpha_1,\alpha_2,\alpha_3) \in\mathbb C(\textsf{i})^4$ with $0< \Re \alpha_{\ell} <1$ for $\ell=0,1,2,3$ and 
set $q=(z_1,z_2) , \ \xi= (\zeta_1,\zeta_2) \in J_a^b$. We have the following identities:   

\begin{enumerate} 
\item $\frac{1}{2} \mathfrak D_{a_1}^{(\alpha_0,\alpha_1)}[f](\xi,q) = \dfrac{d}{d\bar z_1} \mathcal I_a^q [f] (\xi,q,\vec{\alpha})$.   

\item $\frac{1}{2}  \mathfrak D_{a_2}^{(\alpha_2,\alpha_3)}[f](\xi,q)= \dfrac{d}{d\bar z_2} \mathcal I_a^q [f] (\xi,q,\vec{\alpha})   $.

\item $\mathfrak D_{a_1}^{(\alpha_0,\alpha_1)}\circ \mathfrak I_{a_1}^{(\alpha_0, \alpha_1)}[f](\xi,q)= \Re f(\Re z_1 + i \Im \zeta_1, \zeta_2 ) + i \Im f (\Re \zeta_1  + i \Im z_1, \zeta_2).$ 

Particularly, $\mathfrak D_{a_1}^{(\alpha_0,\alpha_1)}\circ \mathfrak I_{a_1}^{(\alpha_0, \alpha_1)}[f](\xi,q)\mid_{q=\xi}= f(\xi)$. 

\item $\mathfrak D_{a_2}^{(\alpha_2,\alpha_3)}\circ  \mathfrak I_{a_2}^{(\alpha_2, \alpha_3)}[f](\xi,q) =\Re f ( \zeta_1 , \Re z_2 + i \Im \zeta_2 ) + i \Im f(\zeta_1 , \Re \zeta_2  + i \Im z_2)$. 

Particularly, $\mathfrak D_{a_2}^{(\alpha_2,\alpha_3)}\circ  \mathfrak I_{a_2}^{(\alpha_2, \alpha_3)}[f](\xi,q)\mid_{q=\xi}= f(\xi)$. 

\item $ \mathfrak D_{a_1}^{(\alpha_0,\alpha_1)}\circ   \mathfrak I_{a_2}^{(\alpha_2, \alpha_3)}[f](\zeta, z) =  \mathfrak D_{a_2}^{(\alpha_2,\alpha_3)}\circ \mathfrak I_{a_1}^{(\alpha_0, \alpha_1)}[f](\zeta, z)=0$

\item $2\dfrac{d}{d z_1} 	\circ \mathfrak D_{a_1}^{(\alpha_0,\alpha_1)}[f](\xi,q) = \Delta_{z_1}  \mathcal I_a^q [f] (\xi,q,\vec{\alpha})$, 
where $\Delta_{z_1}$   denotes the Laplacian operator  in $\mathbb R^2$ according to the real components of $z_1$.

\item $2\dfrac{d}{d z_2}  \mathfrak D_{a_2}^{(\alpha_2,\alpha_3)}[f](\xi,q)= \Delta_{z_2}  \mathcal I_a^q [f] (\xi,q,\vec{\alpha})   $, 
where $\Delta_{z_2}$   denotes the Laplacian operator  in $\mathbb R^2$ according to the real components of $z_2$.

\item  If the mapping $q\to \mathcal I_a^{q} [f](\xi, q,\vec{\alpha})$ belongs to $C^2(J_a^b, \mathbb C(i)(\mathbb C(\textsf{i})))$  for all $\xi$ and set $\vec{\beta} = (\beta_0, \beta_1, \beta_2, \beta_3) \in\mathbb C(\textsf{i})^4$ with $0< \Re(\alpha_{\ell}+\beta_{\ell}) <1$ for $\ell=0,1,2,3$  then we have  
\begin{align*} 
 \mathfrak D_{a_1}^{ ({\alpha}_0, \alpha_1)} \circ  \mathfrak D_{a_1}^{ ({\beta}_0,\beta_1)}[f] (\xi,q) = &   D _{\Re a_1^+}^{{\alpha_0 + \beta_0}} f(\Re z_1 + i \Im \zeta_1 , \zeta_2) -   D _{\Im a_1^+}^{{\alpha_1 + \beta_1}} f(\Re \zeta_1 + i \Im z_1 , \zeta_2) ,\\
 \mathfrak D_{a_2}^{ ({\alpha}_2, \alpha_3)} \circ  \mathfrak D_{a_2}^{ ({\beta}_2,\beta_3)}[f] (\xi,q) =  
 & -  D _{\Re a_2^+}^{{\alpha_2 + \beta_2}} f( \zeta_1, \Re z_2 + i \Im \zeta_2 )-
 D _{\Im a_2^+}^{{\alpha_3 + \beta_3}} f(\zeta_1, \Re \zeta_2 + i\Im z_1 )
\end{align*} 
\begin{align*} 
 \mathfrak D_{a_1}^{ ({\alpha}_0, \alpha_1)} \circ \mathfrak D_{a_2}^{ ({\beta}_2,\beta_3)}[f] (\xi,q)   =  
 \mathfrak D_{a_2}^{ ({\alpha}_2, \alpha_3)} \circ \mathfrak D_{a_1}^{ ({\beta}_0,\beta_1)}[f] (\xi,q) = 0
\end{align*} 
\begin{align*} 
\overline{  \mathfrak D_{a_1}^{ ({\alpha}_0, \alpha_1)} } \circ  \mathfrak D_{a_1}^{ ({\beta}_0,\beta_1)}[f] (\xi,q) = &   D _{\Re a_1^+}^{{\alpha_0 + \beta_0}} f(\Re z_1 + i \Im \zeta_1 , \zeta_2) + D _{\Im a_1^+}^{{\alpha_1 + \beta_1}} f(\Re \zeta_1 + i \Im z_1 , \zeta_2) ,\\
\overline{  \mathfrak D_{a_2}^{ ({\alpha}_2, \alpha_3)}} \circ  \mathfrak D_{a_2}^{ ({\beta}_2,\beta_3)}[f] (\xi,q) =  
 & +  D _{\Re a_2^+}^{{\alpha_2 + \beta_2}} f( \zeta_1, \Re z_2 + i \Im \zeta_2 )+
 D _{\Im a_2^+}^{{\alpha_3 + \beta_3}} f(\zeta_1, \Re \zeta_2 + i\Im z_1 ),
\end{align*} 
where 
$$\overline{\mathfrak D_{a_1}^{ ({\alpha}_0, \alpha_1)}} [f](\xi,q) = 2 \dfrac{d}{d  z_1} \mathcal I_a^q [f] (\xi,q,\vec{\alpha}),$$
$$\overline{\mathfrak D_{a_2}^{ ({\alpha}_2, \alpha_3)}} [f](\xi,q) = 2 \dfrac{d}{d  z_2} \mathcal I_a^q [f] (\xi,q,\vec{\alpha}).$$
\end{enumerate}
\end{proposition}
\begin{proof} Its follows by Proposition \ref{propFRACD1} on account of the above Remarks \ref{remark001} and \ref{remark01}.
\end{proof}

\begin{proposition}(Stokes type integral formula induced by $ {\bf D}_{a}^{\vec{\alpha}}$ and ${\bf  D}_{r,a}^{\vec{\beta}}$).

\noindent  
If $\vec{\alpha}= (\alpha_0, \alpha_1,\alpha_2,\alpha_3), \vec{\beta}= (\beta_0, \beta_1,\beta_2,\beta_3) \in\mathbb C(\textsf{i})^4$ with $0< \Re\alpha_{\ell}, \Re\beta_{\ell}<1$ for $\ell=0,1,2,3$ and  let $f,g \in AC^1(\overline{J_a^b}, \mathbb C(i))$ consider $\xi\in J_a^b$ such that   
 the mappings $q\mapsto  \mathcal I_a^q [f](\xi,q, \vec{\alpha})$ and $ q\mapsto \mathcal I_a^q [g](\xi,q, \vec{\beta} )$ belong to  $ C^1(\overline{J_a^b}, \mathbb C(i)(\mathbb C(\textsf{i})) )$.   
 Then 
 \begin{align*} &   \int_{\partial J_a^b}   \mathcal I_a^q [g](\xi,q, \vec{\beta}) 
  (d \bar q_{[2]} \wedge dq)
 \mathcal I_a^q [f](\xi,q, \vec{\alpha}) \\ 
=  &   2   \int_{J_a^b }  \left(   \mathcal I_a^q [g](\xi,q,\vec{\beta}) \    \mathfrak D_{a_1}^{(\alpha_0,\alpha_1)}[f](\xi,q)    + \   \mathfrak D_{r, a_1}^{(\beta_0,\beta_1)}[g](\xi,q)  \mathcal I_a^q [f](\xi,q,\vec{\alpha})\right)d\mu_q.
\end{align*}
and 
 \begin{align*} &  - \int_{\partial J_a^b}   \mathcal I_a^q [g](\xi,q, \vec{\beta}) 
 ie^{i\theta}(d  q _{[1]} \wedge d\bar  q ) \overline{  \mathcal I_a^q [f](\xi,q, \vec{\alpha}) }\\ 
=  &   2   \int_{J_a^b }  \left(   \mathcal I_a^q [g](\xi,q,\vec{\beta}) \  
i e^{i\theta}\overline{    \mathfrak D_{a_2}^{(\alpha_2,\alpha_3)}[f](\xi,q)} 
 + \     \mathfrak D_{r, a_2}^{(\beta_2,\beta_3)}[g](\xi,q) i e^{i\theta}\overline{\mathcal I_a^q [f](\xi,q,\vec{\alpha})}\right)  d\mu_q.
\end{align*}
\end{proposition}
\begin{proof} Use Proposition \ref{Stokes} and the identities:
\begin{align*} 
{}^{\theta}\mathfrak D_a^{\vec{\alpha}}[f] (\xi, q) =    \mathfrak D_{a_1}^{(\alpha_0,\alpha_1)}[f](\xi,q) 
 + i e^{i\theta}j    \mathfrak D_{a_2}^{(\alpha_2,\alpha_3)}[f](\xi,q)  \end{align*}

\begin{align*} 
{}^{\theta}\mathfrak D_{r,a}^{\vec{\beta}}[g] (\xi, q) =    \mathfrak D_{r, a_1}^{(\beta_0,\beta_1)}[g](\xi,q) 
 +    \mathfrak D_{r, a_2}^{(\beta_2,\beta_3)}[g](\xi,q) i e^{i\theta}j  
\end{align*}
$$\sigma^{\theta}_{ q }:=\frac{1}{2} [d \bar  q _{[2]} \wedge d q  - ie^{i\theta}d  q _{[1]} \wedge d\bar  q  j],$$
\end{proof}

\begin{corollary} Let $\vec{\alpha}= (\alpha_0, \alpha_1,\alpha_2,\alpha_3), \vec{\beta}= (\beta_0, \beta_1,\beta_2,\beta_3) \in\mathbb C(\textsf{i})^4$ with $0< \Re\alpha_{\ell}, \Re\beta_{\ell}<1$ for $\ell=0,1,2,3$ and $f,g \in AC^1(\overline{J_a^b}, \mathbb C(i))$. Consider $\xi\in J_a^b$ such that the mappings $q\mapsto  \mathcal I_a^q [f](\xi,q, \vec{\alpha})$ and $q\mapsto \mathcal I_a^q [g](\xi,q, \vec{\beta})$ belong to $C^1(\overline{J_a^b}, \mathbb C(i)(\mathbb C(\textsf{i}))).$   

If $ \mathfrak D_{a_1}^{(\alpha_0,\alpha_1)}[f](\xi,\cdot) =  \mathfrak D_{r, a_1}^{(\beta_0,\beta_1)}[g](\xi,\cdot) =0$ on  $J_a^b$ then  
\begin{align*}  
\int_{\partial J_a^b}   \mathcal I_a^q [g](\xi,q, \vec{\beta})(d \bar q_{[2]} \wedge dq)\mathcal I_a^q [f](\xi,q, \vec{\alpha}) = 0.
\end{align*}

Similarly, if $\mathfrak D_{a_2}^{(\alpha_2,\alpha_3)}[f](\xi,\cdot) = \mathfrak D_{r, a_2}^{(\beta_2,\beta_3)}[g](\xi,\cdot)=0$ on $J_a^b$ then  
\begin{align*}  
- \int_{\partial J_a^b}   \mathcal I_a^q [g](\xi,q, \vec{\beta}) ie^{i\theta}(d  q _{[1]} \wedge d\bar  q ) \overline{  \mathcal I_a^q [f](\xi,q, \vec{\alpha})}= 0
\end{align*}
\end{corollary}
 
\begin{theorem} (Borel-Pompieu type formula induced by $ {\bf D}_{a}^{\vec{\alpha}}$ and $ {\bf D}_{r,a}^{\vec{\beta}}$) 
Let $\vec{\alpha}=(\alpha_0, \alpha_1, \alpha_2,\alpha_3),\vec{\beta}= (\beta_0, \beta_1, \beta_2,\beta_3)\in\mathbb C(\textsf{i})^4$ with $0< \Re\alpha_{\ell}, \Re\beta_{\ell}<1$ for $\ell=0,1,2,3$ and $f,g \in AC^1(\overline{J_a^b}, \mathbb C(i))$. Consider $\xi\in J_a^b$ such that the mappings $q\to \mathcal I_a^q [f](\xi,q,\vec{\alpha})$ and $q\to \mathcal I_a^q [g](\xi,q, \vec{\beta})$, for $q\in J_a^b$, belong to $C^1(\overline{J_a^b},\mathbb C(i)(\mathbb C(\textsf{i}))).$ Then 
\begin{align*}  &  
  \int_{\partial J_a^b } \left\{\frac{1}{2} [ \  K^{ \alpha_0, \alpha_1}_{a}(\tau,q) ( d \bar \tau_{[2]} \wedge d\tau ) - 
 K^{ \alpha_2, \alpha_3}_{a}(\tau,q)  ie^{-i\theta}\overline{(d \tau_{[1]} \wedge d\bar \tau) } \ ]
 \mathcal I_a^{\tau} [f](\xi,\tau, \vec{\alpha})  \right. \\
& \left.	\hspace{1cm}	+     \mathcal I_a^{\tau} [g](\xi,\tau,\vec{\beta}) \frac{1}{2} [\ ( d \bar \tau_{[2]} \wedge d\tau ) K^{ \beta_0, \beta_1}_{a}(\tau,q)  + ie^{i\theta}(d \tau_{[1]} \wedge d\bar \tau )\overline{ 
 K^{ \beta_2, \beta_3}_{a}(\tau,q) } \ ] \right\}  \\
		&		- 
\int_{J_a^b} \left(   K^{ \alpha_0, \alpha_1}_{a}(\tau,q)   \mathfrak D_{a_1}^{(\alpha_0,\alpha_1)}[f ] (\xi,\tau)  + 
 K^{ \alpha_2, \alpha_3}_{a}(\tau,q)     i e^{-i\theta}     \mathfrak D_{a_2}^{(\alpha_2,\alpha_3)}
	  [f](\xi,\tau) \right. \\
			&  \left. 
			\hspace{1.5cm}+
 \mathfrak D_{r, a_1}^{(\beta_0,\beta_1)}[g](\xi,\tau) K^{ \beta_0, \beta_1}_{a}(\tau,q) 
   -    \mathfrak D_{r, a_2}^{(\beta_2,\beta_3)}[g](\xi,\tau) i e^{i\theta} \overline{  K^{ \beta_2, \beta_3}_{a}(\tau,q) }
    		\right)  d\mu_{\tau}       \\
		=  &    \left\{ \begin{array}{ll}  (f+g) (\Re  z_1 + i\Im\zeta_1, \zeta_2) + 
		(f+g) (\Re  \zeta_1 + i\Im z_1, \zeta_2)  & \\
		+ 
		(f+g) (\zeta_1, \Re  z_2 + i\Im\zeta_2) + (f+g) (\zeta_1, \Re  \zeta_2 + i\Im z_2) & \\
+ N[f](\xi,q, \vec{\alpha}) + N[g](\xi,q, \vec{\beta}) , &    q\in 
		J_a^b ,  \\ 0 , &  q\in \mathbb C(i)^2\setminus\overline{J_a^b},                    
\end{array} \right. 
\end{align*} 
\end{theorem} 
\begin{proof}
Use Theorem \ref{B-P-F-D} and the following identities:

\begin{align*} 
 \mathfrak K^{\vec{\alpha}}_{\theta, a}(\tau,q) {}^{\theta}   \mathfrak D_{a}^{\vec \alpha}[f](\xi,\tau)  = &   K^{ \alpha_0, \alpha_1}_{a}(\tau,q)   \mathfrak D_{a_1}^{(\alpha_0,\alpha_1)}[f ] (\xi,\tau)  + 
 K^{ \alpha_2, \alpha_3}_{a}(\tau,q)     i e^{-i\theta}     \mathfrak D_{a_2}^{(\alpha_2,\alpha_3)}
	  [f](\xi,\tau)  \\ 
		&  +   \left( K^{ \alpha_0, \alpha_1}_{a}(\tau,q)   i e^{i\theta}  \overline{  \mathfrak D_{a_2}^{(\alpha_2,\alpha_3)}
	  [f](\xi,\tau) }  +  K^{ \alpha_2, \alpha_3}_{a}(\tau,q) \overline{ \mathfrak D_{a_1}^{(\alpha_0,\alpha_1)}[f ] (\xi,\tau)}\right)j					
\end{align*}

\begin{align*} 
{}^{\theta}\mathfrak D_{r,a}^{\vec{\beta}}[g] (\xi, \tau)\mathfrak K^{\vec{\beta}}_{\theta, a}(\tau,q)  = & 
  \mathfrak D_{r, a_1}^{(\beta_0,\beta_1)}[g](\xi,\tau) K^{ \beta_0, \beta_1}_{a}(\tau,q) 
   -    \mathfrak D_{r, a_2}^{(\beta_2,\beta_3)}[g](\xi,\tau) i e^{i\theta} \overline{  K^{ \beta_2, \beta_3}_{a}(\tau,q) } \\
	& + \left(    \mathfrak D_{r, a_1}^{(\beta_0,\beta_1)}[g](\xi,\tau)  K^{ \beta_2, \beta_3}_{a}(\tau,q) +   \mathfrak D_{r, a_2}^{(\beta_2,\beta_3)}[g](\xi,\tau) i e^{i\theta} \overline{  K^{ \beta_2, \beta_3}_{a}(\tau,q)}   \right) j
\end{align*}

\begin{align*}	\mathfrak K^{\vec{\alpha}}_{\theta, a}(\tau,q) \sigma^{\theta}_{\tau} =  & 
 \frac{1}{2} [ \  K^{ \alpha_0, \alpha_1}_{a}(\tau,q) ( d \bar \tau_{[2]} \wedge d\tau ) - 
 K^{ \alpha_2, \alpha_3}_{a}(\tau,q)  ie^{-i\theta}\overline{(d \tau_{[1]} \wedge d\bar \tau) } \ ]
 \\
& + \frac{1}{2} [   - K^{ \alpha_0, \alpha_1}_{a}(\tau,q) ie^{i\theta}( d \tau_{[1]} \wedge d\bar \tau) +
   K^{ \alpha_2, \alpha_3}_{a}(\tau,q) \overline{ (d \bar \tau_{[2]} \wedge d\tau ) }   \  ]j  
\end{align*}

\begin{align*}
\sigma^{\theta}_{\tau}	\mathfrak K^{\vec{\beta}}_{\theta, a}(\tau,q) = & \frac{1}{2} [\ ( d \bar \tau_{[2]} \wedge d\tau ) K^{ \beta_0, \beta_1}_{a}(\tau,q)  + ie^{i\theta}(d \tau_{[1]} \wedge d\bar \tau )\overline{K^{ \beta_2, \beta_3}_{a}(\tau,q)} \ ] \\
& + \frac{1}{2} [\ - ie^{i\theta}(d \tau_{[1]} \wedge d\bar \tau) \overline{K^{\beta_0, \beta_1}_{a}(\tau,q)}+ (d \bar \tau_{[2]} \wedge d\tau) K^{\beta_2, \beta_3}_{a}(\tau,q) \ ] j 
\end{align*}
\end{proof}

\begin{corollary} 
Let $\vec{\alpha}= (\alpha_0, \alpha_1,\alpha_2,\alpha_3), \vec{\beta}= (\beta_0, \beta_1,\beta_2,\beta_3)\in\mathbb C(\textsf{i})^4$ with $0< \Re\alpha_{\ell}, \Re\beta_{\ell}<1$ for $\ell=0,1,2,3$ and $f,g \in AC^1(\overline{J_a^b}, \mathbb C(i))$. Let $\xi\in J_a^b$ such that the mappings $q\to \mathcal I_a^q [f](\xi,q,\vec{\alpha})$ and $q\to \mathcal I_a^q [g](\xi,q, \vec{\beta})$, for $q\in J_a^b$, belong to $C^1(\overline{J_a^b}, \mathbb C(i)(\mathbb C(\textsf{i}))).$ Then 

\begin{align*}  &  
  \int_{\partial J_a^b } \left\{\frac{1}{2} [ \  K^{ \alpha_0, \alpha_1}_{a}(\tau,\xi) ( d \bar \tau_{[2]} \wedge d\tau ) - 
 K^{ \alpha_2, \alpha_3}_{a}(\tau,\xi)  ie^{-i\theta}\overline{(d \tau_{[1]} \wedge d\bar \tau) } \ ]
 \mathcal I_a^{\tau} [f](\xi,\tau, \vec{\alpha})  \right. \\
& \left.	\hspace{1cm}	+     \mathcal I_a^{\tau} [g](\xi,\tau,\vec{\beta}) \frac{1}{2} [\ ( d \bar \tau_{[2]} \wedge d\tau ) K^{ \beta_0, \beta_1}_{a}(\tau,\xi)  + ie^{i\theta}(d \tau_{[1]} \wedge d\bar \tau )\overline{ 
 K^{ \beta_2, \beta_3}_{a}(\tau,\xi) } \ ] \right\}  \\
		&		- 
\int_{J_a^b} \left(   K^{ \alpha_0, \alpha_1}_{a}(\tau,\xi)   \mathfrak D_{a_1}^{(\alpha_0,\alpha_1)}[f ] (\xi,\tau)  + 
 K^{ \alpha_2, \alpha_3}_{a}(\tau,\xi)     i e^{-i\theta}     \mathfrak D_{a_2}^{(\alpha_2,\alpha_3)}
	  [f](\xi,\tau) \right. \\
			&  \left. 
			\hspace{1.5cm}+
 \mathfrak D_{r, a_1}^{(\beta_0,\beta_1)}[g](\xi,\tau) K^{ \beta_0, \beta_1}_{a}(\tau,\xi) 
   -    \mathfrak D_{r, a_2}^{(\beta_2,\beta_3)}[g](\xi,\tau) i e^{i\theta} \overline{  K^{ \beta_2, \beta_3}_{a}(\tau,\xi) }
    		\right)  d\mu_{\tau}       \\
		 		=  &      4(f+g) (\zeta_1, \zeta_2) 
	 + N[f](\zeta,\zeta , \vec{\alpha}) + N[g](\zeta,\zeta, \vec{\beta}) . 
\end{align*} 

On the other hand, if $$\mathfrak D_{a_1}^{(\alpha_0,\alpha_1)} [f ] (\xi,\cdot)= 
 \mathfrak D_{a_2}^{(\alpha_2,\alpha_3)}[ f  ] (\xi,\cdot) =
	\mathfrak D_{r,a_1}^{(\alpha_0,\alpha_1)} [g ] (\xi,\cdot) =
	\mathfrak D_{r,a_2}^{(\alpha_2,\alpha_3)}[ g  ] (\xi, \cdot) =0$$ on $J_a^b$, then   

\begin{align*}  &  
  \int_{\partial J_a^b } \left\{\frac{1}{2} [ \  K^{ \alpha_0, \alpha_1}_{a}(\tau,q) ( d \bar \tau_{[2]} \wedge d\tau ) - 
 K^{ \alpha_2, \alpha_3}_{a}(\tau,q)  ie^{-i\theta}\overline{(d \tau_{[1]} \wedge d\bar \tau) } \ ]
 \mathcal I_a^{\tau} [f](\xi,\tau, \vec{\alpha})  \right. \\
& \left.	\hspace{1cm}	+     \mathcal I_a^{\tau} [g](\xi,\tau,\vec{\beta}) \frac{1}{2} [\ ( d \bar \tau_{[2]} \wedge d\tau ) K^{ \beta_0, \beta_1}_{a}(\tau,q)  + ie^{i\theta}(d \tau_{[1]} \wedge d\bar \tau )\overline{ 
 K^{ \beta_2, \beta_3}_{a}(\tau,q) } \ ] \right\}  \\
		=  &    \left\{ \begin{array}{ll}  (f+g) (\Re  z_1 + i\Im\zeta_1, \zeta_2) + 
		(f+g) (\Re  \zeta_1 + i\Im z_1, \zeta_2)  & \\
		+ 
		(f+g) (\zeta_1, \Re  z_2 + i\Im\zeta_2) + (f+g) (\zeta_1, \Re  \zeta_2 + i\Im z_2) & \\
+ N[f](\xi,q, \vec{\alpha}) + N[g](\xi,q, \vec{\beta}) , &    q\in 
		J_a^b ,  \\ 0 , &  q\in \mathbb C(i)^2\setminus\overline{J_a^b},                    
\end{array} \right. 
\end{align*} 
Particularly, 
\begin{align*}  &  
  \int_{\partial J_a^b } \left\{\frac{1}{2} [ \  K^{ \alpha_0, \alpha_1}_{a}(\tau,\xi) ( d \bar \tau_{[2]} \wedge d\tau ) - 
 K^{ \alpha_2, \alpha_3}_{a}(\tau,\xi)  ie^{-i\theta}\overline{(d \tau_{[1]} \wedge d\bar \tau) } \ ]
 \mathcal I_a^{\tau} [f](\xi,\tau, \vec{\alpha})  \right. \\
& \left.	\hspace{1cm}	+     \mathcal I_a^{\tau} [g](\xi,\tau,\vec{\beta}) \frac{1}{2} [\ ( d \bar \tau_{[2]} \wedge d\tau ) K^{ \beta_0, \beta_1}_{a}(\tau,\xi)  + ie^{i\theta}(d \tau_{[1]} \wedge d\bar \tau )\overline{ 
 K^{ \beta_2, \beta_3}_{a}(\tau,\xi) } \ ] \right\}  \\
				=  &       4(f+g) (\xi,\xi) + N[f](\xi,\xi, \vec{\alpha}) + N[g](\xi,\xi, \vec{\beta}) .
\end{align*} 
\end{corollary} 
\section*{Declarations}
\subsection*{Funding} Instituto Polit\'ecnico Nacional (grant number SIP20211188) and CONACYT.
\subsection*{Conflict of interest} The authors declare that they have no conflict of interest regarding the publication of this paper.
\subsection*{Author contributions} Both authors contributed equally to the manuscript and typed, read, and approved the final form of the manuscript, which is the result of an intensive collaboration.
\subsection*{Availability of data and material} Not applicable
\subsection*{Code availability} Not applicable


\begin{thebibliography}{99}
\bibitem{Ba} Baleanu, D.; Restrepo, J. E.; Suragan, D. \textit{A class of time-fractional Dirac type operators}. Chaos Solitons Fractals 143 (2021), 110590, 15 pp.
 
\bibitem{Be} Bernstein, S. \textit{A fractional Dirac operator}. Noncommutative analysis, operator theory and applications, 27--41, Oper. Theory Adv. Appl., 252, Linear Oper. Linear Syst., Birkhäuser/Springer, [Cham], 2016. 

\bibitem{C} Caramalho Domingues, J. \textit{Chapter 20 - S.F. Lacroix, Traité du calcul différentiel et du calcul intégral}, first edition (1797–1800),
Editor(s): I. Grattan-Guinness, Roger Cooke, Leo Corry, Pierre Crépel, Niccolo Guicciardini, Landmark Writings in Western Mathematics 1640-1940,
Elsevier Science, 2005, 277-291.

\bibitem{CC} Contharteze Grigoletto, E.; Capelas de Oliveira, E. \textit{Fractional Versions of the Fundamental Theorem of Calculus}, Appl. Math., 2013, 4, 23--33.  

\bibitem{CDOP} Coloma, N.; Di Teodoro, A.; Ochoa-Tocachi, D.; Ponce, F. \textit{Fractional Elementary Bicomplex Functions in the Riemann–Liouville Sense}. Adv. Appl. Clifford Algebr. 31 (2021), no. 4, Paper No. 63.

\bibitem{DM} Delgado, B.B.; Macías-Díaz, J.E. \textit{On the General Solutions of Some Non-Homogeneous Div-Curl Systems with Riemann–Liouville and Caputo Fractional Derivatives}. Fractal Fract. 2021, 5, 117.

\bibitem{FRV} Ferreira, M,  Kraußhar, R. S.,  Rodrigues, M. M., Vieira, N. \textit{A higher dimensional fractional Borel-Pompeiu formula and a related hypercomplex fractional operator calculus}. Math. Methods Appl. Sci. 42 (2019), no. 10, 3633--3653.

\bibitem{FV1} Ferreira, M., Vieira, N. \textit{Eigenfunctions and fundamental solutions of the fractional Laplace and Dirac operators: the Riemann–Liouville case}. Complex Anal. Oper. Theory 10 (5), 1081-1100 (2016).

\bibitem{FV2} Ferreira, M., Vieira, N. \textit{Eigenfunctions and fundamental solutions of the fractional Laplace and Dirac operators using Caputo derivatives}. Complex Var. Elliptic Equ. 62 (9), 1237-1253, (2017).

\bibitem{GS} Gonz\'alez Cervantes, J. O.; Shapiro, M. \textit{Hyperholomorphic Bergman spaces and Bergman operators associated to domains in $\mathbb C^2$}. Complex Anal. Oper. Theory 2 (2008), no. 2, 361--382. 

\bibitem{BG1}  Gonz\'alez Cervantes; J. O., Bory Reyes, J. \textit{A quaternionic fractional Borel-Pompeiu type formula} Fractal, Vol. 30, No. 1 (2022) 2250013 (15 pages).

\bibitem{GM} Gorenflo, R., Mainardi, F. \textit{Fractional calculus: integral and differential equations of fractional order}. Fractals and fractional calculus in continuum mechanics (Udine, 1996), 223--276, CISM Courses and Lect., 378, Springer, Vienna, 1997.
 	
\bibitem{GS1} G\"urlebeck, K., Spr\"ossig, W. \textit{Quaternionic analysis and elliptic boundary value problems}. Birkha\"user Verlag, 1990.
 	
\bibitem{GS2} G\"urlebeck, K., Spr\"ossig, W. \textit{Quaternionic and Clifford calculus for physicists and engineers}. John Wiley and Sons, 1997.

\bibitem{HG} Heidrich, R.; Jank, G. \textit{On iteration of quaternionic M\"obius transformation}, Compl. Var. Theory Appls., 1996,
vol. 29. 313- 318.

\bibitem{HKZ} Hedenmalm, H.; Korenblum, B.; Zhu, K. \textit{Theory of Bergman spaces, Graduate Texts in Mathematics}, 199. Springer-Verlag, New York, 2000.

\bibitem{KS} Kravchenko, V. V.,  Shapiro, M. V. \textit{Integral Representations for Spatial Models of Mathematical Physics}. Pitman Research Notes in Mathematics Series, Longman: Harlow, 1996.

\bibitem{K} Kravchenko, V.V. \textit{Applied Quaternionic Analysis. Research and Exposition in Mathematics}. Heldermann Verlag: Lemgo, 2003.

\bibitem{KV} K\"ahler, U., Vieira, N. \textit{Fractional Clifford analysis}. In: Bernstein, S., K\"ahler, U., Sabadini, I., Sommen, F. (eds.) Hypercomplex analysis: new perspectives and applications. Trends in mathematics, 191-201. Birk\"ahuser, Basel, 2014.

\bibitem{KST} Kilbas, A. A., Srivastava, H. M., Trujillo, J. J. \textit{Theory and Applications of Fractional Differential Equations}. North-Holland Mathematics Studies, 204. Elsevier Science B.V., Amsterdam, 2006.

\bibitem{MABZ} Magin, R.L.,  Abdullah, O.,  Baleanu, D.,  Zhou, X. J. Anomalous diffusion expressed through fractional order differential operators in the Bloch-Torrey equation, Journal of Magnetic Resonance, Volume 190, Issue 2, 2008, 255-270.

\bibitem{MS} Mitelman, I., Shapiro, M. \textit{Differentiation of the Martinelli-Bochner integrals and the notion of the hyperderivability}. Math. Nachr., 1995, 172, 211-238.

\bibitem{MR} Miller, K. S., Ross, B. \textit{An Introduction to the Fractional Calculus and Fractional Differential Equations}. A Wiley-Interscience Publication. John Wiley \& Sons, Inc., New York, 1993.

\bibitem{Na} Naser, M. \textit{Hyperholomorphic functions}. Siberian Math. J. 12, 959-968, 1971.

\bibitem{No1} N$\hat{o}$no, K. \textit{On the quaternion linearization of Laplacian 1}. Bull. Fukuoka Univ. Ed. III 35, 5-10, 1986.

\bibitem{No2} N$\hat{o}$no, K. \textit{Hyperholomorphic functions of a quaternion variable}. Bull. Fukuoka Univ. Ed. III 32, 21-37, 1983.

\bibitem{OS} Oldham, K. B.,  Spanier, J. \textit{The Fractional Calculus}. Dover Publ. Inc., 2006.

\bibitem{O} Ortigueira, M. D. \textit{Fractional calculus for scientists and engineers}. Lecture Notes in Electrical Engineering, 84. Springer, Dordrecht, 2011. 

\bibitem{PBBB} Peña Pérez, Y. Abreu Blaya, R. Árciga Alejandre, M. P., Bory Reyes, J. \textit{Biquaternionic reformulation of a fractional monochromatic Maxwell system}. Adv. High Energy Phys. 2020, Art. ID 6894580, 9 pp. 

\bibitem{P} Podlubny, I. \textit{Fractional differential equations. An introduction to fractional derivatives, fractional differential equations, to methods of their solution and some of their applications}. Mathematics in Science and Engineering, 198. Academic Press, Inc., San Diego, CA, 1999.


\bibitem{Ro} Ross B. \textit{A brief history and exposition of the fundamental theory of fractional calculus}. In: Ross B. (eds) Fractional Calculus and Its Applications. Lecture Notes in Mathematics,  (1975) vol 457. Springer, Berlin, Heidelberg.

\bibitem{SKM}  Samko, S.G., Kilbas, A.A.,  Marichev,  O.I. \textit{Fractional Integrals and Derivatives. Theory and Applications}. Gordon and Breach Sci. Publ. London, New York, 1993.

\bibitem{S1} Shapiro, M. \textit{Quaternionic analysis and some conventional theories}. In: Alpay, D. (ed.) Operator Theory, 1423-1446. Springer, Basel, 2015.

\bibitem{S2} Shapiro, M. \textit{Some remarks on generalizations of the one dimensional complex analysis: hypercomplex approach}. (Trieste, 1993), in: Functional Analytic Methods in Complex Analysis and Applications to Partial Differential Equations, World Scienti.c Publ., River Edge, NJ, 1995, 379-401.
 	
\bibitem{SV1} Shapiro, M, Vasilevski, N. L. \textit{Quaternionic $\psi$-monogenic functions, singular operators and boundary value problems. I. $\psi$-Hyperholomorphy function theory.}  Compl. Var. Theory Appl. \textbf{27} (1995), 17--46.
 	
\bibitem{SV2} Shapiro, M., Vasilevski, N. L. \textit{Quaternionic $\psi$-hyperholomorphic functions, singular operators  and boundary value problems II. Algebras of singular integral operators and Riemann type boundary value problems.} Compl. Var. Theory Appl. \textbf{27} (1995), 67--96.
 	
\bibitem{sudbery} Sudbery, A. \textit{Quaternionic analysis}, Math. Proc. Phil. Soc. \textbf{85} (1979), 199--225.

\bibitem{T} Tarasov, V. E. \textit{Fractional Dynamics: Applications of Fractional Calculus to Dynamics of Particles, Fields and Media}. Springer, New York, 2011.

\bibitem{VTRMB} Valério, D., Trujillo, J.J., Rivero, M. et al. \textit{Fractional calculus: A survey of useful formulas}. Eur. Phys. J. Spec. Top. 222, 1827-1846, 2013.

\bibitem{V} Vieira, N. \textit{Fischer decomposition and Cauchy-Kovalevskaya extension in fractional Clifford analysis: the Riemann-Liouville case}. Proc. Edinb. Math. Soc. II. 60 (1), 251-272, 2017.
\end{thebibliography}
\end{document}